\documentclass{amsart}

\usepackage{amssymb, graphicx}
\usepackage[leqno]{amsmath}
\usepackage{dsfont}
\usepackage{bbm}

\usepackage{thm-restate}

\usepackage{braket}
\usepackage{enumitem}
\usepackage{hyperref}
\usepackage{cleveref}
\usepackage{setspace}

\usepackage{tikz}        
\usetikzlibrary{matrix, arrows}

\usepackage{xcolor}
\hypersetup{
	colorlinks,
	linkcolor={red!60!black},
	citecolor={blue!60!black},
	urlcolor={blue!80!black}
}

\crefname{subsection}{subsection}{subsections}

\makeatletter
\newcommand{\leqnomode}{\tagsleft@true}
\newcommand{\reqnomode}{\tagsleft@false}
\makeatother

\DeclareFontFamily{U}{mathx}{\hyphenchar\font45}
\DeclareFontShape{U}{mathx}{m}{n}{<-> mathx10}{}
\DeclareSymbolFont{mathx}{U}{mathx}{m}{n}
\DeclareMathAccent{\widebar}{0}{mathx}{"73}

\DeclareMathSymbol{\omicron}{\mathord}{letters}{"6F}

\newtheorem{theorem}{Theorem}[section]
\newtheorem{remark}[theorem]{Remark}
\newtheorem{notation}[theorem]{Notation}
\newtheorem{lemma}[theorem]{Lemma}

\newtheorem{corollary}[theorem]{Corollary}
\newtheorem{proposition}[theorem]{Proposition}
\newtheorem{question}[theorem]{Quesion}
\newtheorem{fact}[theorem]{Fact}

\theoremstyle{definition}
\newtheorem{example}[theorem]{Example}
\newtheorem{definition}[theorem]{Definition}

\newcommand{\twopartdef}[4]
{
	\left\{
	\begin{array}{ll}
		#1 & \mbox{if } #2 \\
		#3 & \mbox{if } #4
	\end{array}
	\right.
}

\newcommand{\cM}{\mathcal{M}}
\newcommand{\cN}{\mathcal{N}}
\newcommand{\cL}{\mathcal{L}}
\newcommand{\cP}{\mathcal{P}}
\newcommand{\cI}{\mathcal{I}}
\newcommand{\cS}{\mathcal{S}}

\newcommand{\cA}{\mathcal{A}}
\newcommand{\cB}{\mathcal{B}}

\newcommand{\cD}{\mathcal{D}}

\newcommand{\bM}{\mathbb{M}}
\newcommand{\bN}{\mathbb{N}}
\newcommand{\bZ}{\mathbb{Z}}

\newcommand{\omegaLomega}{\omega^{<\omega}}
\newcommand{\omegaomega}{\omega^\omega}
\newcommand{\omegaomegastar}{\omega^{\omega^*}}

\newcommand{\vphi}{\varphi}

\newcommand{\scomp}[2]{#1\left[#2\right]}

\newcommand{\scompMNs}{\scomp{\cM}{\cN^s}}
\newcommand{\gcompMN}{\scomp{\cM}{\cN_a}_{a\in\cM}}
\newcommand{\gcompMNs}{\scomp{\cM}{\cN^s_a}_{a\in\cM}}

\newcommand{\into}{\hookrightarrow}

\DeclareMathOperator{\red}{red}
\DeclareMathOperator{\blue}{blue}
\DeclareMathOperator{\tp}{tp}
\DeclareMathOperator{\aut}{Aut}
\DeclareMathOperator{\Th}{Th}

\DeclareMathOperator{\height}{height}
\DeclareMathOperator{\leaf}{leaf}
\DeclareMathOperator{\branch}{branch}
\DeclareMathOperator{\suc}{succ}
\DeclareMathOperator{\In}{int}

\DeclareMathOperator{\rot}{root}

\newcommand{\famTree}{\Braket{T,\cM_t}}
\newcommand{\famTreeBig}{\left\langle T,<,\left( \cM_t \right)_{t\in \In(T)}\right\rangle}

\newcommand{\infProds}{\prod_T \cM^s_t }
\newcommand{\infProdNs}{\prod_T \cN^s_t }

\title{Infinite Lexicographic Products}
\author[Nadav Meir]{Nadav Meir}\thanks{The work in this paper is part of the author's Ph.D. studies at the Department of Mathematics, Ben-Gurion University of the Negev under the supervision of Assaf Hasson. \\ \indent The author was Partially supported by ISF Grant 181/60 and the Hillel Gauchman scholarship.}
\email{mein@math.bgu.ac.il}
\address{Department of Mathematics,
	Ben Gurion University of the Negev \\
	P.O.B. 653, 
	Be'er Sheva 8410501, Israel.
}

\begin{document}
\reqnomode
	\begin{abstract}	
		We generalize the lexicographic product of first-order structures by presenting a framework for constructions which, in a sense, mimic iterating the lexicographic product infinitely and not necessarily countably many times. We then define dense substructures in infinite products and show that any countable product of countable transitive homogeneous structures has a unique countable dense substructure, up to isomorphism. Furthermore, this dense substructure is transitive, homogeneous and elementarily embeds into the product. This result is then utilized to construct a rigid elementarily indivisible structure.		
	\end{abstract}
	\maketitle

\section{Introduction}

Much of mathematics in general deals with the construction of new mathematical structures using existing ones as building blocks. Examples of such constructions are pervasive throughout mathematics. In algebra, there are constructions such as direct products, wreath products, and tensor products. In topology, there is the product topology, a.k.a. the product space. In graph theory and combinatorics, there are numerous notions of a product of two given graphs, such as the Cartesian product, the tensor product, and the lexicographic product. Even in na\"{i}ve set theory, the Cartesian product of sets plays a primary role, and, in axiomatic set theory, the existence of a Cartesian product of infinitely many sets is a source of long standing debates.

The study of properties of the new structures relative to their building blocks is central in each of the mathematical branches mentioned above. Examples span from classification of subgroups of a direct product of groups, through classical results in topology regarding preservation of separation axioms under products, and lack thereof, to 
the calculation of the chromatic number of a product of two graphs by means of arithmetic on their chromatic numbers. Indeed, the author has yet to find a major branch of mathematics in which this is not the case.

In \cite{Me}, the author studied a construction of the same nature, in the context of first-order relational structures, as defined below. 
\begin{definition}\label{defGenProduct}
	Let $\cM$, $\{\cN_a\}_{a\in \cM}$ be structures in a relational language, $\cL$. Let $M, \{N_a\}_{a\in \cM}$ be their universes, respectively.
	The \emph{generalized lexicographic product} $\gcompMN$ is the $\cL$-structure whose universe is $\bigcup_{a\in M} \{a\}\times N_a$ where for every
	$n$-ary relation $R\in \cL$ we set $R^{\gcompMN}$ to be
	\begin{align*}
	& \Set{\big((a,b_1),\dots,(a,b_n)\big)\ |\  a\in M \text{\ \ \ and\ \ \ } \cN_a \models R(b_1,\dots, b_n)  } \ \  \cup \\ 
	& \Set{\big((a_1,b_1),\dots,(a_n,b_n)\big)\ |\ \begin{matrix} \bigvee_{1\leq i\neq j\leq n} a_i\neq a_j \text{\ \ \ \  and \ \ \ \ } \cM \models R(a_1,\dots, a_n) \end{matrix} }.
	\end{align*}
\end{definition}
\noindent In case $\cN=\cN_a$ for all $a\in \cM$, this is abbreviated by $\scomp{\cM}{\cN}$ and this construction generalizes the lexicographic order and the lexicographic product of graphs. More generally, in a binary language, $\cM[\cN]$ coincides with a classical construction denoted by the same notation. (e.g. \cite{Cherlin,Lachlan}.)

Let $\gcompMN^s$ be $\gcompMN$
expanded by a binary relation $s\notin \cL$ interpreted as 
$\Set{\big((a,b_1),(a,b_2)\big) | a\in \cM \text{\ \ and\ \ } b_1,b_2\in \cN_a}$.
Let $\cN_a^s$ is $\cN_a$ expanded by a binary relation $s$ interpreted as $\left(\cN_a\right)^2$ for all $a\in \cM$. Then
$\gcompMN^s = \gcompMNs$. For this reason, we identify the two constructions and denote both by $\gcompMNs$.

In \cite{Me}, the author proved several results demonstrating the good behavior of the lexicographic product with regards to elementary equivalence, elementary embeddings, quantifier elimination, etc. (See \Cref{section:lexicographic results} for a reveiw of some of these results.)
Consequenty, many model-theoretic properties such as simplicity, stability, NIP, etc. are preserved under lexicographic products.

The study of lexicographic products was motivated by several question regarding elementarily indivisible structures, as defined below.  A first-order relational structure is \emph{indivisible} if for every colouring of its universe in two colours, there is a monochromatic substructure isomorphic to it. Additionally, it is \emph{elementarily indivisible} if the monochromatic substructure can be chosen to be elementary. Indivisibility of relational first-order structures is a well studied notion
in Ramsey theory
(e.g., \cite{KoRo86}, \cite{ElSa91}, \cite{ElSa93}, and  \cite[Appendix A]{Fr00}). In \cite{HKO11} by Hasson, Kojman and Onshuus asked three question concerning elementarily indivisible structure, two of them answered in \cite{Me}. The motivation behind the research presented in this paper is the third and final question in \cite{HKO11}:
\begin{question}[\text{\cite[Question 6.7]{HKO11}}]\label{QFinal}
	Is there a rigid elementarily indivisible structure?
	
	\noindent \emph{Here, by \emph{rigid}, we mean a structure whose automorphism group is trivial.}
\end{question}

In this paper, we take the lexicographic product a step further, to infinity (and beyond), 
by presenting a framework for constructions which, in a sense, mimic iterating the lexicographic product infinitely and not necessarily countably many times. 

We then concentrate on the case of countably many iterations, in which we define the notion of a \emph{dense substructure}; we prove that for any countable product of countable transitive homogeneous structures has a unique countable dense substructure, up to isomorphism. Furthermore, this dense substructure is transitive, homogeneous. In addition, we show that such a dense substructure not only elementarily embeds into the product, but also there is an $\cL_{\omega_1,\omega}$-elementary embedding of it into the product.

As an application, we show that a dense substructure in a countable lexicographic products of countable homogeneous indivisible structures is elementarily indivisible and answer \Cref{QFinal}.

We conclude by defining a strengthening of the elementary indivisibility to $\cL_{\omega_1,\omega}$ and show that every $\cL_{\omega_1,\omega}$-elementarily indivisible structure is transitive.

\subsection*{Acknowledgements} The author is grateful to Assaf Hasson for presenting the question which motivated this paper, as well as for the fruitful discussions and the warm support along the way.

\section{Preliminaries}\label{sec:Prelims}

In this \namecref{sec:Prelims} we summarize some of the context for our results, including a few basic definitions from model theory, mainly concentrating on countable infinitary logic, i.e., $\cL_{\omega_1,\omega}$ and its relation to homogeneity, as well as several results from \cite{Me} that this paper generalizes.

Unless otherwise specified, we do not distinguish between a structure $\cM$ and its universe (or underlying set). Throughout this paper all languages are relational, so there is no distinction between subsets and substructures of a given structure. The notation for both is $B\subseteq \cM$. We denote the cardinality of a structure $\cM$ by $|\cM|$.

\subsection{Homogeneity and infinitary logic}

\begin{definition}
	If $\cM$ and $\cN$ are $\cL$-structures and $B \subseteq \cM$, then $f : B \to \cN$ is a
	\emph{partial isomorphism} if
	$\cM \models \vphi\left(\bar{b}\right) \iff \cN\models \vphi\left(f\left(\bar{b}\right)\right)$
	for all \emph{quantifier-free} (or equivalently, atomic) $\cL$-formulas $\vphi$ and all finite sequences $\bar{b}$ from B.
	
\end{definition}

\begin{definition}\label{defUltrahom}
	A structure $\cM$ is \emph{homogeneous} if whenever $A \subset \cM$ with $|A| < |\cM|$ and $f : A \to \cM$ is a partial
	isomorphism, there is an automorphism $\sigma\in \aut(\cM)$ such that $\sigma \upharpoonright A = f$.
\end{definition}

\begin{definition}
	An $\cL_{\omega_1,\omega}$-theory $T$ admits \emph{quantifier elimination} (QE) if for every $\cL_{\omega_1,\omega}$-formula $\phi$ there is a quantifier-free $\cL_{\omega_1,\omega}$-formula $\psi$ such that
	$ T\models \phi \leftrightarrow \psi$. An $\cL$-structure  $\cM$ admits $\cL_{\omega_1,\omega}$-QE if its $\cL_{\omega_1,\omega}$-theory admits QE.
\end{definition}

\begin{definition}
	Let $\cL$ be a first-order language and let $\bar{v}= v_1,\dots, v_n$.
	\begin{enumerate}
		\item An \emph{$\cL$-diagram} in variables $\bar{v}$ is a (perhaps partial) type $p$ consisting of only atomic and negated atomic $\cL$-formulas.
		\item An \emph{$\cL$-diagram} in variables $\bar{v}$ is \emph{complete} if for every $k$-ary $R\in \cL$ and every $1\leq i_1,\dots,i_k\leq n$, either $R(v_{i_1},\dots, v_{i_k})\in p$ or $\neg R(v_{i_1},\dots, v_{i_k})\in p$.
		\item An \emph{$\cL$-diagram} in variables $\bar{v}$ is \emph{$T$-consistent} for $T$, where $T$ is either an $\cL$-theory or an $\cL_{\omega_1,\omega}$-theory if
		$T\not\models \neg\exists \bar{v}\bigwedge_{\phi\in p} \phi(\bar{v})$.
	\end{enumerate}
\end{definition}

\begin{lemma}\label{countableStrNotmanydiagrams}
	Let $\cL$ be a first-order language. If $\cM$ is an $\cL$-structure of size $\kappa\geq \aleph_0$, then for any complete $\cL$-diagram $p$ in variables $\bar{v}$, there is some $\cL$-diagram $q$ of size $\kappa$ in variables $\bar{v}$ such that 
	$\bar{a}\models p \iff \bar{a}\models q$ for all $\bar{a}\in \cM$.
\end{lemma}
\begin{proof}
	Assume not. We construct, by induction, a sequence of pairwise distinct tuples $\Set{\bar{a}_\alpha | \alpha<\kappa^+}\subseteq \cM$  and formulas $\Set{\phi_\alpha | \alpha<\kappa^+}\subseteq p$ such that $\bar{a}_\beta\not\models \phi_\beta$ and $\bar{a}_\beta\models \phi_{\alpha}$ for all $\alpha<\beta<\kappa^+$. This will contradict $|\cM|=\kappa$.
	\begin{itemize}
		\item There is some $\phi_0\in p$ and $\bar{a}_0\in\cM$ such that $\bar{a}_0\not\models \phi_0$.
		\item Assume $\bar{a}_\alpha$ and $\phi_{\alpha}$ were defined for all $\alpha<\beta<\kappa^+$. Since $\beta<\kappa^+$, there is some $\phi_{\beta}\in p$ and $\bar{a}_\beta\in \cM$ such that $\bar{a}_\beta\models \phi_{\alpha}$ for all $\alpha<\beta<\kappa^+$ but $\bar{a}_\beta\not\models \phi_\beta$.
	\end{itemize}
	By the construction, $\bar{a}_\alpha\not\models \phi_{\alpha}$ and $\bar{a}_\beta\models \phi_{\alpha}$ for all $\alpha<\beta<\kappa^+$. so $\bar{a}_\alpha\neq \bar{a}_\beta$.
\end{proof}

\begin{lemma}\label{infinitaryDNF}
	Let $\cL$ be a first-order language, let $\cM$ be a countable $\cL$-structure and let  $\phi(\bar{v})$ be a quantifier-free $\cL_{\omega_1,\omega}$-formula.
	
	If $\phi$ is quantifier-free or $\cM$ is homogeneous, then $\phi$ has a \emph{disjunctive normal form}, i.e., a formula of the form  $\bigvee_{j\in J}\bigwedge_{i\in I_j} \theta_i(\bar{v})$ such that $J$ and $I_j$ are countable for all $j\in J$ and $\theta_i$ is atomic or negated atomic for all $i\in \bigcup_{j\in J} I_j$ and 
	\[ \cM\models \forall \bar{v} \left( \phi(\bar{v})\leftrightarrow  \bigvee_{j\in J}\bigwedge_{i\in I_j} \theta_i(\bar{v}) \right)_. \]
\end{lemma}

\begin{proof}
	Let $\Psi$ be the set of all complete $\cL$-diagrams realized in $\cM$. By \Cref{countableStrNotmanydiagrams}, for any $p\in\Psi$, there is some countable $q_p$ such that $\bar{a}\models p\iff \bar{a}\models q_p$ for all $\bar{a}\in \cM$. Let $\Psi':=\Set{q_p | p\in \Psi}$. Then for any $\bar{a}\in \cM$, there is some $q\in \Psi'$ such that $\bar{a}\models q$. 
	\begin{itemize}
		\item If $\phi$ is quantifier-free, then for any complete $\cL$-diagram $p$, either $p\vdash \phi$ or $p\vdash \neg \phi$. 
		\item If $\cM$ is homogeneous, then for any two tuples $\bar{a},\bar{b}$ satisfying the same complete $\cL$-diagram, there is a partial isomorphism $f:\bar{a}\to \bar{b}$, which in turn extends to an automorphism of $\cM$. So $\bar{a},\bar{b}$ satisfy the same $\cL_{\omega_1,\omega}$-formulas. 
	\end{itemize}
	In conclusion, if either $\phi$ is quantifier free or $\cM$ is homogeneous, then for any complete $\cL$-diagram $p\in \Psi$, we have that  either $\bar{a}\models p \implies \cM\models \phi(\bar{a})$ for all $\bar{a}\in \cM$, or $\bar{a}\models p \implies \cM\not\models \phi(\bar{a})$ for all $\bar{a}\in \cM$. 
	So for any $q\in \Psi'$ we have that either $\bar{a}\models q \implies \cM\models \phi(\bar{a})$ for all $\bar{a}\in \cM$, or $\bar{a}\models q \implies \cM\not\models \phi(\bar{a})$ for all $\bar{a}\in \cM$. Let
	\begin{align*}
	& \Psi_1: = \Set{q \in \Psi' | \bar{a}\models q \implies \cM\models \phi(\bar{a})\text{  for all }\bar{a}\in \cM}, \\
	& \Psi_2: = \Set{q \in \Psi' | \bar{a}\models q \implies \cM\not\models \phi(\bar{a})\text{  for all }\bar{a}\in \cM}.
	\end{align*}
	So $\Psi' = \Psi_1\cup\Psi_2$, therefore, 
	$\cM\models \phi(\bar{a})\iff \bigvee_{q\in \Psi_1 } \bar{a}\models q$
	for all $\bar{a}\in \cM$.
	Since $\cM$ is countable, so is $\Psi_1$ as a set of realized diagrams. Since every $q$ is a countable $\cL$-diagram, we can write  $\bar{a}\models q \iff \cM \models \bigwedge_{\theta\in q} \theta(\bar{a})$ for all $\bar{a}\in \cM$ and the right hand side is a countable disjunction of atomic and negated atomic formulas. In conclusion
	\[ \cM\models \forall \bar{v} \left( \phi(\bar{v})\leftrightarrow \bigvee_{q\in \Psi_1 } \bigwedge_{\theta\in q} \theta_i(\bar{v}) \right) \]
	and $\Psi_1, q$ are countable.
\end{proof}

\begin{lemma}\label{indivisibleHom}
	Let $\cM$ be a countable structure. Then $\cM$ is homogeneous if and only if $\cM$ admits $\cL_{\omega_1,\omega}$-QE, which, in turn, implies that every embedding between isomorphic copies of $\cM$ is elementary.
\end{lemma}

\begin{proof}
		$\Rightarrow$ is by \Cref{infinitaryDNF}.
		
		For $\Leftarrow$: Assume $\cM$ admits $\cL_{\omega_1,\omega}$-QE, let $f:\bar{a}\to \bar{b}$ be a finite partial isomorphism and $c\in \cM$. We want to find some $d\in \cM$ such that $f\cup \Braket{c,d}$ is a partial isomorphism.
		Let $p(\bar{v},x)$ be the complete $\cL$-diagram realized by $\bar{a},c$. By \Cref{countableStrNotmanydiagrams}, there is some countable $\cL$-diagram $q$ equivalent to $p$ in $\cM$. It suffices to show that $\cM\models \exists x\,\bigwedge_{\theta\in q}\theta(\bar{b},x)$. Indeed, $\cM\models \exists x\,\bigwedge_{\theta\in q}\theta(\bar{a},x)$ and by $\cL_{\omega_1,\omega}$-QE, there is some quantifier-free $\cL_{\omega_1,\omega}$-formula $\vphi(\bar{v})$ such that $\cM\models \forall \bar{v}\left( \exists x\,\bigwedge_{\theta\in q}\theta(\bar{v},x)\leftrightarrow\vphi(\bar{v}) \right)$. So $\cM\models \vphi(\bar{a})$ and, since $\vphi$ is quantifier free, $\cM\models \vphi(\bar{b})$ so $\cM\models \exists x\,\bigwedge_{\theta\in q}\theta(\bar{b},x)$.
\end{proof}

\subsection{lexicographic products}\label{section:lexicographic results}

\begin{fact}[\text{\cite[Theorem 2.7]{Me}}]\label{compositionQE}
	Let $\cL$ be a relational language, let $s\notin \cL$ be a binary relation symbol and let $T_1, T_2$ be $\cL$-theories (not necessarily complete). 
	If  $T_1$ and $T_2$ both admit QE and $T_1$ has a transitive model then
	there is an $\cL \cup \{s\}$-theory $T$ (not necessarily complete) admitting QE, such that $\gcompMN^s \models T$ whenever $\cM \models T_1$ and $\{\cN_a\}_{a\in \cM} \models T_2$. 
	
	In particular, if $\cM$ and $\cN$ are $\cL$-structures both admitting QE and $\cM$ is transitive
	then $\cM[\cN]^s$ admits QE.
\end{fact}

\begin{fact}[\text{\cite[Proposition 2.21]{Me}}]\label{finProdIndivisible}
	If $\cM$ and $\cN$ are both indivisible, so is $\cM[\cN]^s$.
\end{fact}

\begin{proposition}\label{FinProdElemEmbed}
	Let $\cM, \{\cN_a\}_{a\in \cM};\ \cM', \{\cN'_a\}_{a\in \cM'}$ be structures in a relational language, $\cL$, such that $\Th(\cM)$ has transitive models. If $\cM\prec \cM'$ and $\cN_a\prec \cN'_a$ for all $a\in \cM$ then $\gcompMN^s \prec \cM'[\cN'_a]^s_{a\in \cM'.}$
\end{proposition}

\begin{proof}
	Consider the Morleyzations $\widehat{\cM}, \{\widehat{\cN_a}\}_{a\in \cM};\ \widehat{\cM'}, \{\widehat{\cN'_a}\}_{a\in \cM'}$ as defined in  in  \cite[Notation 2.19]{Me}. By definition of the Morleyzation, there is an $\widehat{\cL}$-theory $\widehat{T}$ eliminating quantifiers, such that all Morleyzations of $\cL$-structures model $\widehat{T}$. Since $\cM\prec \cM'$ and $\cN_a\prec \cN'_a$ for all $a\in \cM$, it follows that $\widehat{\cM}\prec \widehat{\cM'}$ and $\widehat{\cN_a}\prec \widehat{\cN'_a}$ for all $a\in \cM$. By \Cref{compositionQE}, $\widehat{\cM'}[\widehat{\cN'_a}]_{a\in \widehat{\cM'}}^s$ and $\widehat{\cM}[\widehat{\cN_a}]_{a\in \widehat{\cM}}^s$ both model an $\widehat{\cL}\cup\{s\}$-theory which eliminates quantifiers, so the canonical embedding $\widehat{\cM}[\widehat{\cN_a}]_{a\in \widehat{\cM}}^s \hookrightarrow \widehat{\cM'}[\widehat{\cN'_a}]_{a\in \widehat{\cM'}}^s$ is elementary.
\end{proof}

\begin{proposition}\label{finProdHomogeneous}\ 
	\begin{enumerate}
		\item 	If $\cM$ and $\cN$ are transitive, then $\scompMNs$ is transitive.
		\item	If $\cM$ and $\cN$ are $\kappa$-homogeneous for some cardinal $\kappa$, then $\scompMNs$ is $\kappa$-homogeneous.
	\end{enumerate}
\end{proposition}
\begin{proof}
	\begin{enumerate}
		\item Let $a,b\in \cM,c,d\in \cN$ and $f\in \aut(\cM),g\in \aut(\cN)$ such that $f(a)=b, g(c)=d$. Then, for $F\in \aut\left(\scompMNs\right)$ defined by $F((x,y)) := \left(f(x),g(y)\right)$, we have $F((a,b)) = (c,d)$. 
		\item 
		Let $\lambda<\kappa$ and let $\Braket{(a_i,b_i) | i<\lambda+1}$,$\Braket{(c_i,d_i) | i<\lambda}$ be sequences of elements in $\scompMNs$ such that \[\tp_{\scompMNs}^{qf}\left(\Braket{(a_i,b_i) | i<\lambda}\right) =\tp_{\scompMNs}^{qf}\left(\Braket{(c_i,d_i) | i<\lambda}\right).\]
		We need to find some $(c_{\lambda+1},d_{\lambda+1})$ such that  
		\[\tp_{\scompMNs}^{qf}\left(\Braket{(a_i,b_i) | i<\lambda+1}\right) =\tp_{\scompMNs}^{qf}\left(\Braket{(c_i,d_i) | i<\lambda+1}\right).\]
		By $\kappa$-homogeneity of $\cM$ and $\cN$ there are $c_{\lambda+1}\in \cM$ and $d_{\lambda+1}\in \cN$ such that 
		$\tp_{\cM}^{qf}\left(\Braket{a_i | i<\lambda+1}\right) =\tp_{\cM}^{qf}\left(\Braket{c_i | i<\lambda+1}\right)$
		 and \\ 
		$\tp_{\cN}^{qf}\left(\Braket{b_i | i<\lambda+1}\right) =\tp_{\cN}^{qf}\left(\Braket{d_i | i<\lambda+1}\right)$.
		By definition of $\scompMNs$, we are done.
		\end{enumerate}
\end{proof}

\newcommand{\meet}{\land}

\section{Finite tree products}\label{sec:infitinite products}

We can iterate the product defined in \Cref{defGenProduct} any finite number of times, and this product is, in fact, \emph{associative}: using the bijection $(a,(b,c))\mapsto ((a,b),c)$, we get
$\cM\left[\cN\left[\cP\right]^{s_2}\right]^{s_1} \cong \left(\cM\left[\cN\right]^{s_1}\right)\left[\cP\right]^{s_2}$ and 
\[\cM\left[\cN_a\left[\cP_b\right]_{b\in \cN_a}^{s_2}\right]^{s_1}_{a\in \cM}  \cong \left(\cM\left[\cN_a\right]_{a\in M}^{s_1}\right)\left[\cP_b\right]_{{b\in \cM\left[\cN_a\right]}^{s_2}_{a\in \cM}.}^{s_1}\]

If $\cI$ is a structure whose universe is a singleton and $\cI, \cM\models \forall x \neg R(x,\dots,x)$ for all $R\in \cL$, then 
$ \cM[\cI] \cong \cI[\cM] \cong \cM. $

Next, we demonstrate how any finitely iterated product as above is equivalent to a product induced by a tree of finite height.
Consider the example of $\cM\left[\cN_a\left[\cP_b\right]_{b\in \cN_a}^{s_2}\right]^{s_1}_{a\in \cM}$. If we assume, for simplicity, that all structures in the product are structures on $\omega$ as their underlying set. Observe the tree illustrated below.\\

\includegraphics[width=\textwidth]{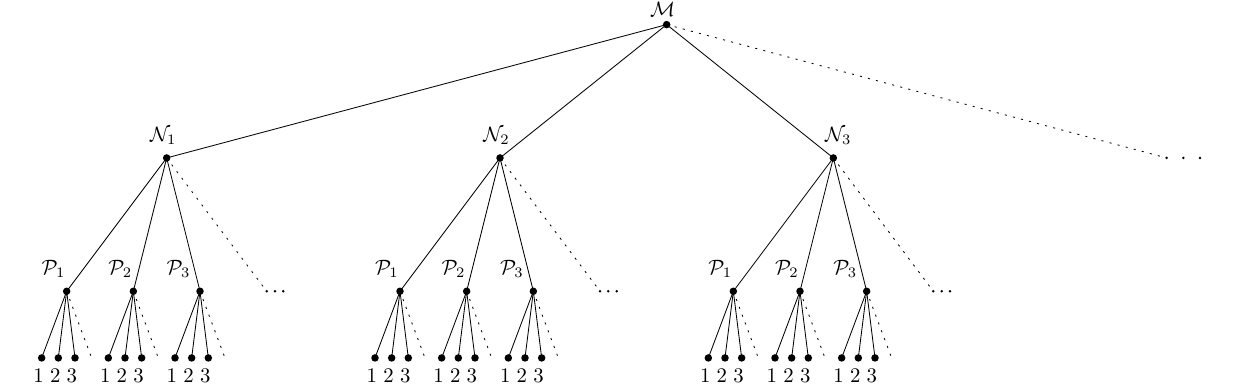}

In this tree, each internal (non-leaf) node of the tree is associated with a structure, $\cS(t)$. For such a node $t$, the set of immediate successors of $t$ are indexed by the universe of $\cS(t)$. Thus, to any node $t \in  T$ (except the root) is associated a unique element, $e(t)$, of the structure inhabiting its immediate predecessor.

For every $k$-tuple of leaves of the tree $(a_1,\dots, a_k)$ such that $\bigvee_{1\leq i< j\leq n} a_i\neq a_j$ we can find some node $m$ in the tree such that $m$ is the \emph{meet} of $a_1,\dots, a_k$, i.e. $m = a_1\meet \dots \meet a_k:= \max\Set{x | x\leq a_1,\dots, a_k}$. Notice that every chain in the tree is discretely-ordered, and thus, $m$ has an immediate successor in the segment $[m,a_i]:=\Set{x | m\leq x\leq a_i}$; call it $S_{a_i}(m)$.
So in the tree products, for every $k$-ary relation $R\in \cL$,
\[ R(a_1,\dots, a_k) \Leftrightarrow \cS(m)\models R(e(S_{a_1}(m)),\dots e(S_{a_k}(m))) \]
and we denote
$s_i(a,b) \Leftrightarrow \height(\meet(a,b))\geq i.$
Notice that the tree product described above is isomorphic to  $\cM\left[\cN_a\left[\cP_b\right]_{b\in \cN_a}^{s_2}\right]^{s_1}_{a\in \cM.}$

In the same sense as above, any finitely iterated product is isomorphic to a product induced by a tree of finite height, defined below.

\begin{definition}\label{defFinTreeProd}	
	 Let $\langle T,<\rangle$ be a tree of finite height, where:
	
	\begin{itemize}
		\item $\leaf(T)$ is the set of $<$-maximal elements in $T$.
		\item $\suc(t):=\Set{s\in T | t<s\ \land\, \not\exists x(t<x<s)}$ for $t\in T$.
		\item $\height(t)$ is order type of the set $\Set{s\in T | s<t}$.
		\item $\height(T):=\max_{t\in T}(\height(t))$.
	\end{itemize}
	
	If  $\left( \cM_t \right)_{t\in T\setminus\leaf(T)}$ is a family of structures in a relational language $\cL$ indexed by $T$, such that each $\cM_t$ is a structure whose universe is $\suc(t)$, then we define the tree product $\prod_T \cM_t$ to be the $\cL$-structure whose universe is $\leaf(T)$ where for every $k$-ary relation $R\in \cL$ we set $R^{\prod_T \cM_t}$ to be \[
	\Set{ (a_1,\dots, a_k) | \cM_m\models R(S_{a_1}(m),\dots, S_{a_k}(m)) \text{ where } m=a_1\land \dots\land a_k }. 
	\]
	
	 	If $\Braket{ s_\alpha | 1\leq \alpha < \height(T) }$ is a sequence of pairwise distinct binary relation symbols disjoint from $\cL$, let $\prod_T \cM^s_t$ be an expansion of $\prod_T \cM_t$ to\\ $\cL\, \cup\, \Set{ s_i |1\leq i<\height(T)}$, where $s_\alpha$ is interpreted as
	$\Set{(a,b) | \height(a\land b)\geq i  }$.

\end{definition}
\begin{remark}\label{oneIsIsoToSucc}
	Let $\cM$ be an $\cL$-structure such that $\cM\models \forall x\neg R(x,\dots, x)$ for every $R\in \cL$.
	Let $T:=\Set{r}\cup \cM$ such that $\cM=\suc(r)$ and let $\cM_r:=\cM$. Then $\cM\cong \cM[\cI] \cong \prod_{T}\cM_t$.
\end{remark}
	By finite induction, results from \cite{Me} such as \Cref{compositionQE} and \Cref{FinProdElemEmbed,finProdHomogeneous} easily extends to tree products where $\height(T)$ is finite. In the following \namecref{secInfTrees}, we generalize some of these results to the case where $T$ may be of infinite height.

\section{Infinite tree products}\label{secInfTrees}

In this \namecref{secInfTrees}, we rigorously defining an infinite iteration process as a product of a tree of structures, not necessarily of finite height.
In \Cref{subsecCountableProducts}, we concentrate on the case where the tree is countable; beforehand, we define and study some of the basic properties of a product induced by a \emph{successor meet tree} of any size or height, defined below.

\begin{definition}\label{def:succesorMeetTree}
	A \emph{successor meet tree} is a partially ordered set $\Braket{T, <}$, such that the following hold:
	\begin{enumerate}
		\item For all $t\in T$, the set $T_{<t}:=\Set{s\in T | s<t}$ is a chain.
		\item For every maximal chain $C\subseteq T$ and $a\in C$, if $a$ is not maximal, then $a$ has an immediate successor in $C$: there is some $s\in C$ such that $a<s$ and for all $s'\in C$, if $a<s'$ then $s\leq s'$.
		\\
		\noindent \emph{We denote the immediate successor $s$ of $a$ in $C$ by $S_C(a)$.
		}
		\item Every $a,b\in T$ have a \emph{meet} $m\in T$: there is some $m\leq a,b$ such that for all $m'\in T$, if $m'\leq a,b$ then $m'\leq m$.\\
		\noindent \emph{We denote the meet $m$ of $a$ and $b$ by $a\meet b$}.
	\end{enumerate}
\end{definition}

\begin{notation}
	Let $\langle T,<\rangle$ be a successor meet tree.
	
	\begin{itemize}
		\item $\branch(T)$ is the set of maximal $<$-chains.
		\item $\leaf(T)$ is the set of $<$-maximal elements in $T$.
		\item $\In(T) = T\setminus \leaf(T)$.
		\item $\suc(t):=\Set{s\in T | t<s\ \land\, \not\exists x(t<x<s)}$ for $t\in T$.
		\item $T_{<t}:= \Set{s\in T | s< t }$ for $t\in T$.
		\item $T_{<A}:= \Set{s\in T | \exists a\in A\, (s< a) } = \bigcup_{a\in A}T_{<a}$ for $A\subseteq T$.
		
		Similarly we define $T_{\leq t}$, $T_{>t}$, $T_{\geq t}$\ ;\ \ $T_{\leq A}$, $T_{>A}$, $T_{\geq A}$.
	\end{itemize}
\end{notation}

\begin{remark}
	If $T$ is a successor meet tree, then so is $T_{\geq t}$ for all $t\in T$, and if $A\subset T$ is a maximal anti-chain, then $T_{\leq A}$ is a successor meet tree as well.
\end{remark}

\begin{lemma}\label{meetBranches}
	Let $\braket{T,<}$ be a successor meet tree.
	\begin{enumerate}
		\item If $B\in \branch(T)$ and $b\in B$,  then $T_{<b}\subset B$.
		\item If $B,C\in \branch(T)$ such that $B\neq C$, then there is some $t\in \In(T)$ such that $B\cap C = T_{\leq t}$.  \emph{We denote such $t$ by $B\meet C$.}
		Moreover, if $b\in B\setminus C$ and $c\in C\setminus B$ then $B\meet C = b\meet c$.

	\end{enumerate}
\end{lemma}

	\begin{proof}
		\begin{enumerate}
			\item\label{BranchInitialSegment} Otherwise, by maximality of $B$, there is some $c\in B$ and $a\in T_{<b}$ such that $c\nleq a$ and $a\nleq c$. Therfore, $b\nleq c$, so $c\in T_{<b}$, contradicting $T_{<b}$ being a chain.
			
			\item By maximality, there are $b\in B\setminus C$ and $c\in C\setminus B$. We claim that $B\cap C = T_{\leq b\meet c}$.
			Indeed, It follows from \Cref{BranchInitialSegment} that $T_{\leq b\meet c}\subseteq B\cap C$. 
			To prove $T_{\leq b\meet c}\supseteq B\cap C$, if there is some $a\in B\cap C\setminus T_{\leq b\meet c}$, then since $B\cap C$ is a chain, $a>b\meet c$ and therefore $b,c\in T_{<a}$. By \Cref{BranchInitialSegment}, $B\notin T_{<c}$ and $c\notin T_{<b}$, contradicting $T_{<a}$ being a chain.
			Finally, $b\meet c\in \In(T)$ since $\left(b\meet c\right)<b,c$.
		\end{enumerate}
	\end{proof}

\begin{definition}\label{defInfProd}\ 
	\begin{enumerate}
		\item Let $\langle T,<\rangle$ be a successor meet tree. If  $\left( \cM_t \right)_{t\in \In(T)}$ is a family of structures in a relational language $\cL$ indexed by $T$, such that each $\cM_t$ is a structure whose universe is $\suc(t)$, then we call $\famTreeBig$ a \emph{family tree}. (abbreviated by $\famTree$)
		\item 	If $\famTree$ is a family tree, we define the product $\prod_T \cM_t$ to be the $\cL$-structure whose universe is $\branch(T)$ where for every $k$-ary relation $R\in \cL$ we set $R^{\prod_T \cM_t}$ to be
		\[ \Set{ (a_1,\dots, a_k) | \cM_m\models R(S_{a_1}(m),\dots, S_{a_k}(m)) \text{ where } m=a_1\land \dots\land a_k }.\]
	\end{enumerate}

\end{definition}

\begin{example}\label{subsectionOmegaLOmega}

	Recall that $\omegaLomega$ is the set of all functions $f:n\to \omega$ for some natural number $n$, and $\omegaomega$ is the set of all functions from $\omega$ to $\omega$.
	For each $a\in \omegaomega$, let $a\upharpoonright n$ be the restriction of $a$ to $n$, which is in $\omegaLomega$.
	We consider the order on $\omegaLomega$ induced by inclusion of functions, i.e. for $t,s\in \omegaLomega$, we define $t\leq s$ if there is some $n\in \omega$ such that $t = s\upharpoonright n$.
	This is indeed a partial order, and, in fact, a \emph{successor meet tree}.
	The following illustrates the order on $\omegaLomega$, where the maximal chains in the order are precisely the elements of $\omegaomega$.
	
	\includegraphics*[width=\textwidth]{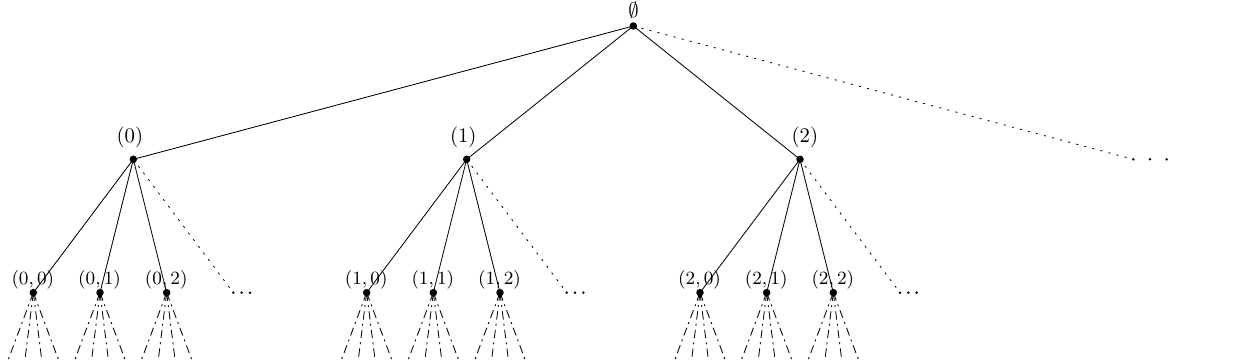}
	
		Let $\left(\cM_t\right)_{t\in \omegaLomega}$ be a family of countable structures in a relational language $\cL$. We assume that for all $t\in \omegaLomega$, the universe of $M_t$ is $\suc(t)$.
		The product $\prod_{\omegaLomega} \cM_t$ is the $\cL$-structure whose universe is $\omegaomega$ where for every $a_1,\dots, a_k\in \omegaomega$ and every $k$-ary relation $R\in \cL$, let $n\in\omega$ be maximal such that $a_1 \upharpoonright n = \dots = a_k \upharpoonright n =: s$. Then
		\[\prod_{\omegaLomega} \cM_t \models R(a_1,\dots, a_k) \iff
		 \cM_s \models R(a_1(n+1),\dots, R(a_k(n+1)).\]
\end{example}

\begin{example}\label{subsectionOmegaOmegaStar}

Recall that $\omega^*$ is the set of natural numbers endowed with the reverse ordering, i.e. $\dots <^* 2 <^* 1 <^* 0$.
For the purposes of this paper, we identify $\Braket{\omega^*,<}$ with the set of negative integers, endowed with the standard linear order on the integers, i.e., $\omega^* = \Set{-1,-2,-3,\dots}$ and $\dots<-3<-2<-1$.
In \Cref{subsectionOmegaLOmega}, we took, as the index set for the family of structures, all initial segments of $\omegaomega$, which turn out to be $\omegaLomega$.
Here we take all initial segments of the set $B:=\Set{a\in \omegaomegastar | a\text{ has finite support}}$.
For every $b\in B$, an \emph{initial segment} of $b$ is of the form $b \upharpoonright \Set{n\in \omega^* | n< m}$ for some $m\in \omega^{*}$.
Now, we define $S$ to be all initial segments of $B$, i.e.:
\[S:= \Set{ a | \exists b\in B, m\in \omega^*\ \Big( a = b\upharpoonright \Set{n\in \omega^* | n<m}\Big) }  \]
We endow $S$ with an order, similar to that of $\omegaLomega$:
$a\leq b \Leftrightarrow a\sqsubseteq b$
where $\sqsubseteq$ is the relation stating $a$ is an initial segment of $b$.
$S$ with this order is a successor meet tree as well.
The following illustrates the order on $S$. In this case, the maximal chains coincide with the maximal elements in the order, which are the elements of $B$. \\

\includegraphics*[width=\textwidth]{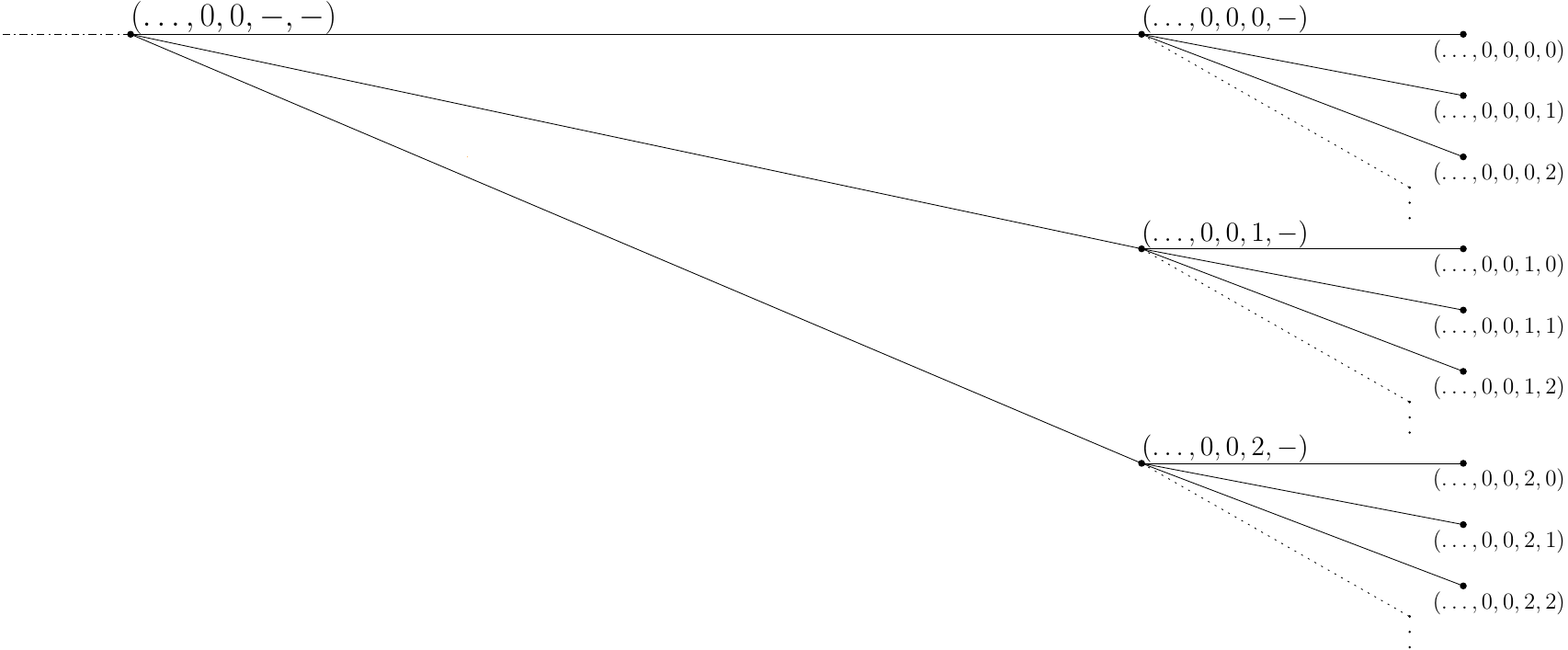}

Let  $\Braket{ \cM_s | s\in \In(S) }$ be a family of countable structures in a relational language where for all $s\in \In(S)$, the universe of $\cM_s$ is $\suc(s)$.
	As in \Cref{subsectionOmegaLOmega}, the product $\prod_S \cM_s$ is the $\cL$-structure whose universe is $B$ where for every $a_1,\dots, a_k\in B$ and every $k$-ary relation $R\in \cL$, let $n\in\omega^*$ be maximal such that $a_1 \upharpoonright n = \dots = a_k \upharpoonright n =: s$. Then
 	\[\prod_S \cM_s \models R(a_1,\dots, a_k) \iff 
 	\cM_s \models R(a_1(n+1),\dots, R(a_k(n+1)).\]
\end{example}

\begin{lemma}\label{countableIsomorphic} Let $T_1, T_2$ be successor meet trees, such that $|\suc(t)|=\aleph_0$ for any $t\in \In(T_1)\cup\In(T_2)$.
	\begin{enumerate}
		\item If the order type of all branches in $T_1$ and $T_2$ is $\omega$ then $T_1\cong T_2$.
		\item If the order type of all branches in $T_1$ and $T_2$ is $\omega^*$ then $T_1\cong T_2$.
		\item If the order type of all branches in $T_1$ and $T_2$ is $\bZ$ then $T_1\cong T_2$.
	\end{enumerate}
\end{lemma}

\begin{proof}
	In all three cases the proof goes as follows. Let $B\in\branch(T_1), C\in \branch(T_2)$.
	We construct, by induction a sequence of subsets $A_0\subseteq A_1\subseteq A_2 \subseteq \dots\subseteq T_1$ such that $\bigcup_{i<\omega}A_i = T_1$, and a sequnec of partial isomorphisms $f_n:A_n\to T_2$ such that $f_0\subseteq f_1\subseteq f_2\subseteq \dots$ and for every $i<\omega$ and $a\in A_i$, either $\suc(a)\subseteq A_i$ and $\suc(f_i(a))\subseteq f_i(A_i)$, or $|\suc(a)\setminus A| = |\suc(f_i(a))\setminus f_i(A)| = \aleph_0$. So $\bigcup_{i<\omega} f_i$ will be an isomorphism.
	
	\begin{itemize}
		\item By the assumption, there is an order isomorphism $f_0:B\to C$ and let $A_0:=B$.
		\item Let $n<\omega$ and assume $f_n:A_n\to T_2$ is a partial isomorphism as in the induction hypothesis. For all $a\in A_n$, such that $|\suc(a)\setminus A_n| = \aleph_0$, let $\Braket{s_i(a) :i<\omega}$ and $\Braket{t_i(a) : i<\omega}$ be enumerations of $\suc(a)$ and $\suc(f_n(a))$, respectively. Then we define 
		\[f_{n+1}: = f_n \cup \Set{ (s_i(a),t_i(a)) : a\in A_n,\ |\suc(a)\setminus A_n| = \aleph_0 ,\ i<\omega}.\]
	\end{itemize}
	Finally, if all branches of $T_1$ are of order type $\omega$, $\omega^*$, or $\bZ$, then for any $t\in T_1$, there is some $b\in B$ and $n\in \bN$ such that $t\in \suc^n(b)$, so $\bigcup_{i<\omega}A_i = T_1$.
\end{proof}

\begin{definition}\label{def:TreeIsomorphism}
	A \emph{family tree isomorphism} between family trees $\famTree$ and $\Braket{U,\cN_u}$ is a bijective, order preserving function $\theta:U\to T$ such that $\theta \upharpoonright \cN_u :\cN_u \to \cM_{\theta(u)}$ is an isomorphism of $\cL$-structures for all $u\in U$.
	
	If there is such an isomorphism, then $\famTree$ and $\Braket{U,\cN_u}$ are \emph{isomorphic}, denoted by $\famTree\cong\Braket{U,\cN_u}$.
\end{definition}

\begin{remark}\label{TreeIsomorphismProdIsomorphism}
	If $\famTree$ and $\Braket{U,\cN_u}$ are isomorphic, then $\prod_T \cM_t\cong \prod_U \cN_u$.
\end{remark}

\begin{definition}
	Let $S$ be a successor meet tree and let $\Braket{ T_B | B\in \branch(S) }$ be a family of successor meet trees indexed by the branches of $S$. Then we define $S*\Braket{ T_B | B\in \branch(S) }$ to be the set $\In(S)\cup \bigcup_{B\in \branch(S)} T_B$ with an order $<$ defined by
	\begin{align*}
	& \Set{(a,b) | a,b \in \In(S) \text{ and } a<b} \cup \\
	& \Set{(a,b) | a,b \in T_B\text{ for some $B\in\branch(S)$ and } a<b} \cup \\
	& \Set{(a,b) | a\in B\cap\In(S) ,\, b \in T_B\text{ for some }B\in\branch(S)}.
	\end{align*}
\end{definition}

\begin{remark}\label{decompositionAntiChain}
	If $T$ is a successor meet tree and $A\subset T$ is a maximal anti-chain, then $A\cap B = \Set{\sup(B)}$ for all $B\in \branch(T_{\leq A}) $ and 
	\[ T = T_{\leq A}* \Braket{ T_{\geq \sup(B)} | B\in \branch(T_{\leq A}) }. \]
\end{remark}

\begin{remark}
	For any $n\in \bZ$, the set $T^{\leq n}:=\Set{t\in T | \height(t)=n}$ is a maximal anti-chain.
\end{remark}

\begin{corollary}\label{decompositionHeight}
	 $\height(\sup(B)) = n$ For any $B\in \branch(T^{\leq n})$ and
		 \[T = T^{\leq n}*\Braket{T_{\geq \sup(B)} | B\in \branch(T^{\leq n})}.\]
\end{corollary}

\begin{lemma}\label{decompositionFin}	Let $S$ be a successor meet tree and let $\Braket{ T_B | B\in \branch(S) }$ be a family of successor meet trees, such that for any $B\in \branch(S)$, if $B$ has a maximal element, then $T_B$ has a minimal element.
	\begin{enumerate}
		\item  $T:=S*\Braket{ T_B}_{B\in \branch(S) }$ is a successor meet tree.
		\item Populating $T$, so that $\famTree$ is a family tree,
		so are $\Braket{S, \cM_s}$ and  $\Braket{T_B, \cM_t}$ for all $B\in \branch{S}$. 
		Furthermore, there is an isomorphism \[f: \prod_T \cM_t\cong \prod_S \cM_s \left[ \prod_{T_B} \cM_t \right]_{B\in \branch\left({S}\right)}   \]
		such that $f\left(\Set{D\in \prod_T \cM_t | D\cap T_B\neq \emptyset}\right) =  \Set{B}\times \prod_{T_B}\cM_t$ for all $B\in \branch(S)$. Moreover, if $D\cap T_B\neq \emptyset$, then $f(D) = (B,D\cap T_B)$.
	\end{enumerate}
	
\end{lemma}

\begin{proof}
	\begin{enumerate}
		\item Exercise.
		\item We will define an isomorphism $g: \prod_S \cM_s \left[ \prod_{T_B} \cM_t \right]_{B\in \branch{S}}	\cong \prod_T \cM_t$ and the wanted $f$ will be $g^{-1}$.
		Let $g$ be defined by $(B,C)\mapsto \left(B\cap \In(S)\right)\cup C$ for all $B\in \branch(S), C\in \branch(T_B)$. Clearly $g((B,C))$ is a chain. 
		
		To show maximality of $g((B,C))$, if $a\in T$ such that $a\leq b$ or $b\leq a$ for all $b\in g((B,C))$, then either $a\in T_B$ which implies $a\in C$, or $a\in \In(S)$ which implies $a\in B$. 
		
		Clearly $g$ is injective. To prove subjectivity, for any $D\in \branch(T)$, there is some $B\in \branch(S)$ such that $D\cap T_B\neq \emptyset$. It follows that $D\cap T_B$ is a maximal chain in $T_B$, so $g((B, D\cap T_B)) = D$. In conclusion $g$ is bijective.
		
		To prove $g$ is an isomorphism. Let $R\in \cL$ be a $k$-ary relation. Let $(B_1,C_1),\dots, (B_k,C_k)\in \prod_S \cM_s \left[ \prod_{T_B} \cM_t \right]_{B\in \branch{(S)}.}$
		Then exactly one of the following cases holds:
		\begin{itemize}
			\item  $\bigwedge_{1\leq i<j\leq k} B_i = B_j$, in which case $C_1,\dots, C_k\in \branch(T_{B_1})$, so $T_{B_1} \ni m:=C_1\land \dots\land C_k = g(B_1,C_1)\land \dots\land g(B_k,C_k)$ and
		\begin{align*}
		 & \prod_S \cM_s \left[ \prod_{T_B} \cM_t \right]_{B\in \branch{S}}\models R((B_1,C_1),\dots, (B_k,C_k)) \iff \\ 
		 & \prod_{T_{B_1}} \cM_t \models R(C_1,\dots,C_k) \iff \\
		 & \cM_m \models R(S_{C_1}(m),\dots, S_{C_k}(m)) \iff \\
		 & \cM_m \models R(S_{g\left((B_1,C_1)\right)}(m),\dots, S_{g\left((B_k,C_k)\right)}(m)) \iff \\
		 & \prod_T \cM_t \models R({g\left((B_1,C_1)\right)},\dots, {g\left((B_k,C_k)\right)}).
		\end{align*}

		\item $\bigvee_{1\leq i<j\leq k} B_i \neq B_j$, in which case $S \ni m:=B_1\land \dots\land B_k = g(B_1,C_1)\land \dots\land g(B_k,C_k)$ and
		\begin{align*}
		& \prod_S \cM_s \left[ \prod_{T_B} \cM_t \right]_{B\in \branch{S}}\models R((B_1,C_1),\dots, (B_k,C_k)) \iff \\ 
		& \prod_S \cM_s \models R(B_1,\dots,B_k) \iff \\
		& \cM_m \models R(S_{B_1}(m),\dots, S_{B_k}(m)) \iff \\
		& \cM_m \models R(S_{g(B_1,C_1)}(m),\dots, S_{g(B_k,C_k)}(m)) \iff \\
		& \prod_T \cM_t \models R({g(B_1,C_1)},\dots, {g(B_k,C_k)}).
		\end{align*}
		\end{itemize}		
		Finally, to prove $g\left(\Set{B}\times \prod_{T_B}\cM_t\right) = \Set{D\in \prod_T \cM_t | D\cap T_B\neq \emptyset}$ for all $B\in \branch(S)$,  If $B\in \branch(S), C\in \prod_{T_B}\cM_t$, let $D: = g((B,C)) = \left(B\cap \In(S)\right)\cup C$ and $D\cap T_{B'}=\left[\left(B\cap \In(S)\right)\cup C\right]\cap T_{B'} = C\cap T_{B'}$, and the latter is non-empty if and only if $C\in \prod_{T_B'}\cM_t$ which happens exactly when $B=B'$. In fact, if $C\cap T_{B}\neq \emptyset$, then $C\subseteq T_{B}$, therefore $D\cap T_{B}=C\cap T_{B} = C$. So if $D\cap T_B\neq \emptyset$, then $f(D)=(B,D\cap T_B)$.
	\end{enumerate}
\end{proof}

\begin{lemma}\label{infProdElemEmbed1}
	Let $T$ be a successor meet tree, $\famTree$, $\Braket{T,\cN_t}$ family trees, and $t_0\in T$, such that $\cM_{t_0}\prec\cN_{t_0}$ and $\cM_t= \cN_t$ for all $t_o\neq t\in T$. Then $\prod_T\cM_t \prec \prod_T\cN_t$.
\end{lemma}

\begin{proof}
	Let $A\subset T$ be any maximal anti-chain such that $t_0\in A$. By \Cref{decompositionAntiChain} and \Cref{decompositionFin}, we have
	\begin{align*}
	& \prod_T\cM_t \cong \prod_{T_{\leq A}} \cM_t \left[ \prod_{T_{\geq \sup(B)}} \cM_t \right]_{B\in \branch(T_{\leq A})} \\
	& \prod_T\cN_t \cong \prod_{T_{\leq A}} \cN_t \left[ \prod_{T_{\geq \sup(B)}} \cN_t \right]_{B\in \branch(T_{\leq A})}.
	\end{align*}
	Let $\cP:=\prod_{T_{\leq A}} \cM_t$, $\cP':=\prod_{T_{\leq A}} \cN_t$, $\cS_B:= \prod_{T_{\geq \sup(B)}} \cM_t$, and $\cS'_B:= \prod_{T_{\geq \sup(B)}} \cN_t$ for all $B\in \branch(T_{\leq A})$.
	So $\prod_T \cM_t\cong \cP[\cS_B]_{B\in \cP}$ and $\prod_T \cN_t\cong \cP'[\cS'_B]_{B\in \cP}$.
	Now there is some $B_0\in \branch(T_{\leq A})$ such that $\{t_0\} = B_0\cap A$ and $t_0=\sup(B_0)$.  Since $\cM_t=\cN_t$ for all $t\neq t_0$, it follows that \[\cP = \prod_{T_{\leq A}}  \cM_t = \prod_{T_{\leq A}} \cN_t = \cP'\] and also \[\cS_B = \prod_{T_{\geq \sup(B)}} \cM_t = \prod_{T_{\geq \sup(B)}} \cN_t =\cB'_B\] for all $B_0\neq B\in \branch(B)$. If $\prod_{T_{\geq t_0}} \cM_t \prec \prod_{T_{\geq t_0}} \cN_t$,
	then $\cS_{B_0}\prec \cS'_{B_0}$ and,  by \Cref{FinProdElemEmbed},
	\[ \prod_T \cM_t \cong \cP'[\cS'_B]_{B\in \cP}\prec \cP[\cS_B]_{B\in \cP}\cong \prod_T \cN_t. \]
	So it suffices to show $\prod_{T_{\geq t_0}} \cM_t \prec \prod_{T_{\geq t_0}} \cN_t$. If $T_{\geq t_0}$ is a tree of finite height, then the claim follows from \Cref{FinProdElemEmbed}.
	 Otherwise, let $S:=T_{\leq t_0}$ and let $A':=\suc(t_0)$. Notice that $A'$ is a maximal anti-chain in $S$ and $t_0\in S_{\leq A'}$. Moreover, $S_{\leq A'}$ is a tree of finite height. Again, by \Cref{decompositionAntiChain}, we have 
	 	\begin{align*}
	 & \prod_S\cM_t \cong \prod_{S_{\leq A'}} \cM_t \left[ \prod_{S_{\geq \sup(B)}}\cM_t \right]_{B\in \branch(S_{\leq A'})} \\
	 & \prod_S\cN_t \cong \prod_{S_{\leq A'}} \cN_t \left[ \prod_{S_{\geq \sup(B)}}\cN_t \right]_{B\in \branch(S_{\leq A'})}.
	 \end{align*}
	 In this case, $t_0\in S_{\leq A'}$, therefore $\prod_{S_{\geq \sup(B)}}\cM_t = \prod_{S_{\geq \sup(B)}}\cN_t$ for all $B\in \branch(S_{\leq A'})$. So it suffices to show that $\prod_{S_{\leq A'}} \cM_t\prec \prod_{S_{\leq A'}} \cN_t$, but this follows from \Cref{oneIsIsoToSucc}.
\end{proof}

\begin{corollary}\label{infProdElemEmbed}
	Let $T$ be a successor meet tree, and let $\famTree$, $\Braket{T,\cN_t}$ be family trees.
	\begin{enumerate}
		\item\label{infProdElemEmbed11} If  $\cM_{t}\prec\cN_{t}$ for all $t\in T$, then $\prod_T\cM_t \prec \prod_T\cN_t$.
		\item\label{infProdElemEmbed22} If $\cM_t\equiv \cN_t$ for all $t\in T$, then $\prod_T\cM_t \equiv \prod_T\cN_t$. 
	\end{enumerate}
\end{corollary}
\begin{proof}
	\Cref{infProdElemEmbed11} is by \Cref{infProdElemEmbed1} and induction. For \Cref{infProdElemEmbed22}, for each $t\in T$, let $\bM_t$ be a sufficiently saturated model of $\Th(\cM_t)$. Then there are elementary embeddings $\cM_t, \cN_t\into \bM_t$. By \Cref{infProdElemEmbed11} of this \namecref{infProdElemEmbed}, we can find elementary embeddings $\prod_T\cM_t , \prod_T\cN_t\into \prod_T\bM_t$.
\end{proof}

\section{Countable  tree products}\label{subsecCountableProducts}

In this \namecref{subsecCountableProducts} we restrict ourself to the case where the structures in the product, as well as the trees themselves, are all countable. 
Furthermore, as in the case of trees of finite height, we assume for simplicity all successor meet trees are leveled, i.e., any two branches have the same order type.

\begin{remark}\label{countableCases}
If $\langle T,<\rangle$ is a leveled countable tree of infinite height such that $|\suc(a)|=\aleph_0$ for all $a\in \In(T)$, then exactly one of the following holds.
\begin{enumerate}
	\item\label{caseOmega} Any branch is of order type $\omega$, which by \Cref{countableIsomorphic}, is isomorphic to \Cref{subsectionOmegaLOmega}.
	\item\label{caseOmegaStar} Any branch is of order type $\omega^*$, which by \Cref{countableIsomorphic}, is isomorphic to \Cref{subsectionOmegaOmegaStar}.
	\item\label{caseZ} Any branch is of order type $\bZ$, which  by \Cref{countableIsomorphic} and \Cref{decompositionFin} is isomorphic to a finite product of the first two cases.
\end{enumerate}
\end{remark}
In each of the cases above, there is a canonical definition of \emph{height} for any element of $T$, as follows:

\begin{definition}
	If $T=\omegaLomega$ or $T=S$ from \Cref{subsectionOmegaOmegaStar} and $t\in T$, then $\height(t)$ is defined to be the maximum of the domain of $t$; i.e., if $t = (7,3,2,8)\in \omegaLomega$ then $\height(t) = 3$, if $t = (0,\dots, 0, 3, 2 , 17, -, -, -)\in S$ then $\height(t) = -4$. For $t = ()\in \omegaLomega$ we set $\height(t):= -1$.

	If $T=S*\Braket{T_B | B\in \branch(S)}$ where $S$ is as in \Cref{subsectionOmegaOmegaStar} and $T_B= \omegaLomega$, then $\height(t)$ is well defined and furthermore, $\height(t_1)<\height(t_2)$ for all $t_1<t_2\in T$ and $\suc(\height(t_1)) = \height(t_2) \iff t_2\in \suc(t_1)$.
\end{definition}
We can now expand any countable product by infinitely many equivalence relations, in the same fashion as in \Cref{defGenProduct}:

\begin{definition}
	Let $\famTree$ be a family tree. Let $\left(\prod_T\cM_t\right)^s$ be an expansion of $\prod_T \cM_t$ by binary relation symbols $\Set{s_n | n\in \bZ}$ interpreted as:
	\[ s_n(x,y)\iff \height(x\meet x)\geq n. \]
\end{definition}

\begin{remark}
	Let $\famTree$ be a family tree. Let $\cM^s_t$ be $\cM_t$ expanded by binary relation symbols $\Set{s_n | n\in \bZ}$ interpreted as:
	\[ \left(s_n\right)^{\cM_t^s}=\twopartdef{\left({\cM_t}\right)^2}{\height(t)\geq n}{\emptyset}{\height(t)< n.} \]
	Then $\left(\prod_T\cM_t\right)^s \cong \prod_T \cM^s_t$. 
	
	\noindent \emph{For this reason, we identify the two constructions and denote the two by $\prod_T \cM^s_t$.}
\end{remark}

\begin{remark}\label{qfTypeChar}
	If $\cM_t$ is transitive for all $t\in T$ and $\bar{a}$ is a tuple in $\prod_T \cM^s_t$ then
	\begin{align*}
	& \tp^{qf}(\bar{a}) = \tp^{qf}(\bar{a}) \upharpoonright \Set{s_n | n\in \bZ}\  \cup\\ & \bigcup\Set{\tp^{qf}_{\cM_m}\left(S_{a_1}(m),\dots,S_{a_k}(m)\right) | a_1,\dots,a_k\in \bar{a}\text{ and } m = a_{1}\meet\dots \meet a_{k}  }_.
	\end{align*}
\end{remark}

\subsection{Dense substructures in countable tree products}\label{secDence}

If $\famTree$ is a family tree where every branch in $T$  is of order type $\omega^*$ (e.g., \Cref{subsectionOmegaOmegaStar}), then $\prod_T\cM_t$ is countable. On the other hand, if every branch in $T$ is of order type $\omega$, then $|\prod_T\cM_t|=2^{\aleph_0}$. In order to keep the size of a product of any countable tree of countable structures to be countable, we introduce the notion of a dense substructure. A dense substructure may be countable, and as an induced substructure will be elementarily equivalent to the product, as will follow from \Cref{denseElemEquiv}. The main result of this \namecref{secDence}, \Cref{denseUnique}, states that under certain homogeneity assumptions on the structures of a countable family tree $\famTree$ dense substructures of the product are homogeneous, and unique up to isomorphism; the precise assumption on $\famTree$ is that it is \emph{pure}, as defined in \Cref{defWpure}.

\begin{definition}
	Let $\famTree$ be a family tree. A substructure $\cD\subseteq \prod_{t\in T} \cM^s_t$ is \emph{dense}
	if for all $t\in T$, there is some $d\in \cD$ such that $d\ni t$.
	
	Clearly whenever $T$ is countable there is a countable dense substructure.
\end{definition}

\begin{remark}\label{denseSubtree}
	Let $\famTree$ be a family tree and let $A\subset T$ be a maximal anti-chain. 
	A substructure  $\cN\subseteq \prod_{t\in T} \cM^s_t$ is dense
		 if and only if 
	for all $a\in A$ and for all $t\geq a$, there is some $d\in \cN$ such that $d\ni t$.
\end{remark}

\begin{remark}\label{omegaStarnonrelevant}
	If every branch in $T$ is of order type $\omega^*$, and $\famTree$ is a family tree with $\cN\subseteq \prod_{t\in T} \cM^s_t$ dense then $\cN = \prod_{t\in T} \cM^s_t$.
\end{remark}

Before continuing, we define a special kind of family tree that will be central throughout this \namecref{secDence}:

\begin{definition}\label{defWpure}
		A family tree $\famTree$  is \emph{pure} if $\cM_t$ is transitive and homogeneous for all $t\in T$ and  $\height(t) = \height(u)\implies \cM_t \cong \cM_u$ for all $t, u\in T$. It is $\omega$-pure if, in addition, branches in $T$ are of order type $\omega$.
\end{definition}

\begin{lemma}\label{BAFstep}
	Let $\famTree$ be a {pure} family tree. If  $\cN\subseteq \prod_T\cM_t^{{s}}$ is a countable dense substructure, then:
	\begin{enumerate}
		\item \label{BAFstep:ultrasat} For any countable $A\subseteq \infProds$, there is $A'\subseteq \cN$ such that $A\cong A'$.
		\item  \label{BAFstep:ultrahom} $\cN$ is  transitive and homogeneous.
	\end{enumerate}
	
\end{lemma}

\begin{proof}
	Let $\bar{a}, b\in \prod_T\cM_t^{{s}}$ and $\bar{c}\in \cN$ where $\bar{a},\bar{c}$ are finite tuples and $\tp^{qf}(\bar{a}) = \tp^{qf}(\bar{c})$. To prove both \Cref{BAFstep:ultrasat} and \ref{BAFstep:ultrahom}, it suffices to find some $d\in \cN$ such that $\tp^{qf}(\bar{a},b) = \tp^{qf}(\bar{c},d)$. If $\bar{a}=\bar{c}=\emptyset$ then by \Cref{qfTypeChar}, for any $d\in \cN$, the mapping $b\mapsto d$ is a partial isomorphism. Otherwise, let $f:\bar{a}\to \bar{c}$ be a partial isomorphism. Let
	$ t_0:=\max\Set{a\meet b | a\in \bar{a}} $. Notice that unless $b\in\bar{a}$, in which case the proof is trivial, $t_0$ exists, as a maximum of finite elements in the chain $b$. 
	Let $\height(t_0) = m$. Let $A_0:= \Set{a\in \bar{a} | a\ni t_0}$. Notice that $A_0$ is the $s_m$-equivalence class of $b$ in $\bar{a}$ and $A_0\neq \emptyset$. Then $f(A_0)$ is also an $s_m$ equivalence class in $\bar{c}$, so there is some $t_1\in T$ with $\height(t_1)=m$ such that $f(A_0) = \Set{c\in\bar{c} | c\ni t_1}$. Since $\cM_{t_0} \cong \cM_{t_1}$ and $\cM_{t_1}$ is homogeneous, it follows that there is some $s\in \cM_{t_1}$ such that
		\begin{align*}
		& \tp^{qf}_{\cM_{t_1}}\left(s, \Set{S_{f(a)}(t_1) | a\in A_0}\right) = 
		 \tp^{qf}_{\cM_{t_0}}\left(S_b(t_0), \Set{S_{a}(t_0) | a\in A_1}\right). 
		\end{align*}
		By density of $\cN$, there is some $d\in \cN$ such that $d\ni s$ and and therefore, by \Cref{qfTypeChar}, $\tp^{qf}(\bar{a},b) = \tp^{qf}(\bar{c},d)$.
\end{proof}

\begin{theorem}\label{denseUnique}
	Let $\famTree$ be a {pure} family tree.
	\begin{enumerate}
		\item Up to isomorphism, there is a unique countable dense substructure $\cD\subseteq \infProds$.
		\item\label{denseUniqueHom} Such a $\cD$ is transitive and homogeneous.
	\end{enumerate}
	\begin{proof}
		Let $\cN_1, \cN_2 \subseteq \infProds$ be two countable dense substructures. By \Cref{BAFstep}, they are both transitive homogeneous, so to prove both 1 and 2 it is left to show that $\cN_1\cong \cN_2$. For that, by \Cref{BAFstep}, every substructure $A\subseteq \cN_1$ is embeddable in $\cN_2$ and vice-versa. Using this fact and homogeneity, a standard back-and-forth argument yields an isomorphism between $\cN_1$ and $\cN_2$.
	\end{proof}
\end{theorem}

\begin{corollary}\label{denseIso}
	Let $\famTree$ and $\Braket{U,\cN_u}$ be isomorphic pure trees.
	If $\cD_1, \cD_2$ are countable dense substructures  in $\infProds, \prod_{U}\cN_u^s$ respectively, then $\cD_1\cong \cD_2$.
\end{corollary}

\begin{proof}
	By \Cref{TreeIsomorphismProdIsomorphism}, $\infProds \cong \prod_{U}\cN^s_u$, so $\cD_2$ is isomorphic (via the restriction of an isomorphism) to a dense substructure of $\infProds$, which in turn, by \Cref{denseUnique}, is isomorphic to $\cD_1$.
\end{proof}

\begin{corollary}\label{denseElemEquiv}
	Let $\famTree$ be a {pure} family tree.
	If $\cD_1,\cD_2\subseteq \infProds$ are dense then $\cD_1\equiv_{{\omega_1,\omega}} \cD_2$.
\end{corollary}

\begin{proof}
	Let $\cD_{01}\subseteq \cD_1$, $\cD_{02}\subseteq \cD_2$ be countable dense substructures. By downwards L\"{o}wenheim-Skolem for $\cL_{\omega_1,\omega}$, there are countable $\cA_1,\cA_2$ such that $\cD_{01}\subseteq \cA_1\preceq_{{\omega_1,\omega}} \cD_1$ and  $\cD_{02}\subseteq \cA_2\preceq_{{\omega_1,\omega}} \cD_2$. Since $\cD_{01}$ and $\cD_{02}$ are dense, so are $\cA_1$ and $\cA_2$. Therefore, by \Cref{denseUnique}, $\cA_1\cong \cA_2$. In conclusion, 
	$ \cD_1\succeq_{{\omega_1,\omega}} \cA_1 \cong \cA_2 \preceq_{{\omega_1,\omega}} \cD_2.  $
\end{proof}

\begin{notation}
	For $\cL$-structures $\cM$ and $\cN$, we denote  $\cM \sim_e \cN$ if $\cM$ can be elementarily embedded in $\cN$ and vice-versa.
\end{notation}

\begin{lemma}\label{allElemEmbed}
	Let $\famTree$ and $\Braket{T,\cN_t}$ be family trees such that $\famTree$ is {pure}.
	If $\cN_{t} \succeq \cM_{t}$ for all $t\in T$, then for any countable dense substructure $\cD_1 \subseteq \prod_T \cM^s_t$ there is some countable dense elementary substructure $\cD_2\preceq  \prod_T \cN^s_t$ such that $\cD_1$ embeds elementarily into $\cD_2$.
\end{lemma}
\begin{proof}
	By downwards L\"{o}wenheim-Skolem, there is some countable dense elementary substructure $\cD'_1\prec \prod_T \cM^s_t$. By \Cref{denseUnique}, $\cD_1\cong \cD'_1$, so we may assume $\cD_1\prec \prod_T \cM^s_t$. Now by \Cref{infProdElemEmbed}, there is an elementary embedding $e: \prod_T \cM^s_t \hookrightarrow \prod_T \cN^s_t$. Again, by L\"{o}wenheim-Skolem, there is some countable dense elementary substructure $e(D_1)\subseteq \cD_2\prec \prod_T \cN^s_t$. So if $\iota$ is the inclusion map we have the following commutative diagram:
	
	\begin{tikzpicture}
\matrix (m) [matrix of math nodes,row sep=1em,column sep=1em,minimum width=1em]
{\prod_T \cM^{s}_t & \ \ \ \ \  & \prod_T \cN^{s}_t &  & \\ \ & \ & \ & \\  \cD_1 & & \cD_2 &   \\ };
\path[right hook->] 
(m-1-1) edge node [above,sloped, allow upside down] {$e$} (m-1-3)

(m-3-1) edge node [left] {$\iota$} node [below, sloped] {$\prec$} (m-1-1)
edge node [above,sloped] {$e$} (m-3-3)
(m-3-3) edge node [left] {$\iota$} node [below, sloped] {$\prec$} (m-1-3);
\end{tikzpicture}

So $e:\cD_1\into \cD_2$ is elementary.
\end{proof}

\begin{lemma}\label{allallElemEmbed}
	Let $\famTree$ and $\Braket{T,\cN_t}$ be family trees such that $\famTree$ is {pure}.
	If $\cN_{t} \preceq \cM_{t}$ for all $t\in T$, then for any countable dense substructure $\cD_1 \subset \prod_T \cM^s_t$ and any countable dense elementary substructure $\cD_2\prec  \prod_T \cN^s_t$ there is an elementary embedding of  $\cD_2$ into $\cD_1$.
\end{lemma}
\begin{proof}
	By \Cref{infProdElemEmbed}, there is an elementary embedding $e: \prod_T \cN^s_t \into \prod_T \cM^s_t$. By L\"{o}wenheim-Skolem, there is some countable dense elementary substructure $e(D_2)\subseteq \cD'_1\prec \prod_T \cM^s_t$. So if $\iota$ is the inclusion map we have the following commutative diagram:
	
	\begin{tikzpicture}
	\matrix (m) [matrix of math nodes,row sep=1em,column sep=1em,minimum width=1em]
	{\prod_T \cM^{s}_t & \ \ \ \ \  & \prod_T \cN^{s}_t &  & \\ \ & \ & \ & \\  \cD'_1 & & \cD_2 &   \\ };
	\path[left hook->] 
	(m-1-3) edge node [above] {$e$} (m-1-1)
	(m-3-3) edge node [above,sloped] {$e$} (m-3-1);
	\path[right hook->] 
	(m-3-1) edge node [left] {$\iota$} node [below, sloped] {$\prec$} (m-1-1)
	(m-3-3) edge node [left] {$\iota$} node [below, sloped] {$\prec$} (m-1-3);
	\end{tikzpicture}
	
	So $e:\cD_2\into \cD'_1$ is elementary. Now by \Cref{denseUnique}, $\cD'_1\cong \cD_1$.
\end{proof}

\begin{corollary}\label{allallElemEmbedBoth}
	Let $\famTree$ and $\Braket{T,\cN_t}$ be family trees such that $\famTree$ is {pure}.
If $\cN_{t} \sim_e \cM_{t}$ for all $t\in T$, then for any countable dense substructure $\cD_1 \subset \prod_T \cM^s_t$ there is some countable dense elementary substructure $\cD_2\prec  \prod_T \cN^s_t$ such that $\cD_1\sim_e\cD_2$.
\end{corollary}

\begin{proof}
	By \Cref{allElemEmbed}, we can find some countable dense elementary substructure $\cD_2\prec \prod_T \cN^s_t$ such that $\cD_1$ elementarily embeds into $\cD_2$. By \Cref{allallElemEmbed}, $\cD_2$ elementarily embeds into $\cD_1$.
\end{proof}

\begin{lemma}\label{densePartition}
	Let $\famTree$ be a countable family tree and $A\subset T$ a maximal anti-chain. Then there is an isomorphism \[ f: \prod_T \cM^s_t \cong \prod_{T_{\leq A}}\cM^s_t \left[ \prod_{T_{\geq \sup(B)}} \cM^s_t \right]_{B\in \branch(T_{\leq A})} \]
	such that 
	\begin{enumerate}
		\item $f\left( \Set{D\in \cN | D\cap T_{\geq \sup(B)}\neq \emptyset} \right) =  f(\cN) \cap \left(\Set{B}\times \prod_{T_{\geq \sup(B)}}\cM^s_t\right)$ for all $\cN\subset \infProds,\ B\in \branch(T_{\leq A})$.
		\item\label{denseThenProdDense} If $\cD\subseteq \prod_T \cM^s_t$ is dense then for every $B\in \branch(T_{\leq A})$, there is a dense $\cD_B\subseteq \prod_{T_B} \cM_t$ such that
		$f(\cD) = \prod_{T_{\leq A}} \cM_t \Big[ D_B \Big]_{B\in \branch(T_{\leq A})} $
		
		\item\label{denseThenProdDenseInv} Conversly, if $\cD_B\subseteq \prod_{T_B} \cM_t$ is dense for every $B\in \branch(T_{\leq A})$, then \[f^{-1}\left(\prod_{T_{\leq A}} \cM_t \Big[ D_B \Big]_{B\in \branch(T_{\leq A})}\right)\] is dense in $\prod_T \cM^s_t$.
	\end{enumerate}
\end{lemma}
\begin{proof}
	Let $f$ be the isomorphism provided by \Cref{decompositionFin}.
	\begin{enumerate}
		\item follows immediately from the definition of $f$.
		
		So for every $\cD\subset \prod_T\cM^s_t$ and every $B\in\branch\left({T_{\leq A}}\right)$ there is some $\cD_B \supseteq \prod_{T_{\geq \sup(B)}}\cM_t^s$ such that \[f\left(\Set{d\in \cD | d\cap T_{\geq \sup(B)}\neq \emptyset}\right) =\Set{B}\times \cD_B.\] So $f(\cD)= \prod_{T_{\leq A}} \cM_t \Big[ \cD_B \Big]_{B\in \branch(T_{\leq A}).}$

		To prove both \ref{denseThenProdDense} and \ref{denseThenProdDenseInv}, by \Cref{denseSubtree}, $\cD$ is dense if and only if for any $a\in A$ and $t\geq a_0$, there is some $d\in\cD$ such that $d\ni t$.
		
		\item Assume $\cD_B$ is dense for all $B\in \branch(T_{\leq A})$. For  any $a_0\in A$ and $t\geq a_0$,  there is some $B_0\in \branch(T_{\leq A})$ such that $a_0=\sup(B_0)$. In particular, $t\in T_{B_0}$ and by density of $\cD_{B_0}$, there is some $C_0\in \cD_{B_0}$ such that $C_0\ni t$. Now let $d:=f^{-1}(B_0,C_0)$. Then $C_0 = d\cap T_{B_0}$. in particular, $t_0\in C_0\subseteq d$.
		\item If $\cD$ is dense, given $t\in T_B$, by density of $\cD$, there is some $d\in \cD$ such that $d\ni t$. Now let $(B,C):=f(d)$. Then $C\in \cD_B$ and $C=d\cap T_B\ni t$.
	\end{enumerate}
\end{proof}

\begin{corollary}\label{denseSub}
	Let $\famTree$ be a countable family tree and let $\cD\subseteq \infProds$ be dense. Then for any $t_0\in T$, the substructure induced on $\cD_{\ni t_0}:=\Set{d\in \cD | d\ni t_0}$ is isomorphic to some dense substructure $\cD_0$ of $\prod_{T_{\geq t_0} \cM_t}$.
\end{corollary}
\begin{proof}
	Let $A$ be a maximal anti-chain such that $t_0\in A$, then there is some $B_0\in \branch(T_{\leq A})$ such that $t_0=\sup(B_0)$. Notice that $\cD_{\ni t_0} = \Set{d\in \cD | d\cap T_{\geq \sup(B_0)}\neq \emptyset}$. 
	
	Let  \[f: \prod_T \cM^s_t \cong \prod_{T_{\leq A}}\cM^s_t \left[ \prod_{T_{\geq \sup(B)}} \cM^s_t \right]_{B\in \branch(T_{\leq A})}\] be an isomorphism provided by \Cref{densePartition}. Then for every $B\in \branch(T_{\leq A})$, there is a dense $\cD_B\subseteq \prod_{T_B} \cM_t$ such that
	$f(\cD) = \prod_{T_{\leq A}} \cM_t \Big[ D_B \Big]_{B\in \branch(T_{\leq A})} $. Thus
	\begin{align*}
	f\left(\cD_{\ni t_0} \right) =  f\left(\cD \right)& \cap \left(\Set{B_0}\times \prod_{T_{\geq \sup(B_0)}}\cM^s_t\right) = \\
	\prod_{T_{\leq A}} \cM_t \Big[ D_B \Big]_{B\in \branch(T_{\leq A})}  & \cap \left(\Set{B_0}\times \prod_{T_{\geq \sup(B_0)}}\cM^s_t\right) = & \Set{B_0}\times \cD_{B_0}  \cong  \cD_{B_0}. \\
	\end{align*}
\end{proof}

\section{(Elementary) indivisibility of infinite tree products}\label{sec:elemIndInf}
 Recall a first-order relational structure is \emph{elementarily indivisible} if for every colouring of its universe in two colours, there is a monochromatic elementary substructure isomorphic to it.

The aim of this \namecref{sec:elemIndInf} is to prove the following theorem, and to utilize it to give an example of a rigid elementarily indivisible structure, giving a negative answer to \Cref{QFinal}. This, together with \cite{Me}, completes answering all questions from \cite{HKO11}.

\begin{theorem}\label{infProdIndOmega}
	Let $\famTree$ be an {$\omega$-pure} family tree, where $\cM_t$ is indivisible for all $t\in T$.
		If  $\cD\subset \infProds$ be a countable dense substructure, then $\cD$ is elementarily indivisible.
\end{theorem}
\begin{proof}
	By \Cref{denseUnique}, \Cref{denseUniqueHom}, $\cD$ is homogeneous, so by \Cref{indivisibleHom} indivisibility and elementary indivisibility coincide. To prove indivisibility, let $c:\cD\to \{\red,\ \blue \}$. By \Cref{denseIso}, it suffices to find a subtree $S\subset T$ and a tree isomorphism $\theta: S\to T$, such that $\theta \upharpoonright \cM_s : \cM_s \to \cM_{\theta(s)}$ is an isomorphism of $\cL$ structures, and a countable dense monochromatic substructure $\cD_2\subset \prod_{S} \cM_t$. 
	For every $t\in T$, let $\cD_{\ni t}:=\Set{a\in \cD | a\ni t}$. So $c$ induces a sub-colouring of $\cD_{\ni t}$. We colour $T$ as follows: 
	\[C(t):=\twopartdef{\blue}{\cD_{\ni t}\text{ contains an isomorphic monochromatic-blue copy of itself.}}{\red}{not.}\]
		If $C(\rot(T)) = \blue$ then we are done. Otherwise, we continue constructing a $C$-red $S$ and $\theta:S\to T$ by induction on $\height(t)$:
	\begin{enumerate}
		\item $S_0 = \rot(T) ;\ \theta_0 = (\rot(T),\rot(T))$.
		\item Assume $C(t)=\red$ for all $t\in S_{n}$ and let $s\in S_n$. by indivisibility of $\cM_s$, either $B(s):=\Set{t\in \suc(s) | C(t) = \blue}$ or $R(s):=\Set{t\in \suc(s) | C(t) = \red}$ contains an isomorphic copy of $\cM_s$. 
	\end{enumerate}
\begin{itemize}
	\item If $B(s)$  contains an isomorphic copy of $\cM_s$, denote it by $\cM'_s$, then $\cD_{\ni u}$ contains an isomorphic monochromatic-blue copy $\cD'_{\ni u}$ of itself for every $u\in \cM'_s$.
	
	By \Cref{denseSub}, for every $u\in T$, there is some dense substructure $\cD_u\subseteq \prod_{T_{\geq u}}\cM_t$ such that  $\cD_u\cong\cD_{\ni u}$. Let $S:=T_{\geq s}$. By \Cref{densePartition} and \Cref{denseUnique}, 
	\begin{align}\label{NmB}
	\cD_{s} \cong \prod_{S_{\leq \suc(s)}} \cM_t^s \left[ \cD_{\ni \sup (B)} \right]_{B\in \branch(S_{\leq \suc(s)})}
	\end{align}
		$\sup(B)\in \suc(s)$ for every $B\in \branch(S_{\leq \suc(s)})$.
		By \Cref{oneIsIsoToSucc}, $\prod_{S_{\leq \suc(s)}} \cM_t^s \cong \cM_s^s$, so together with \Cref{NmB},
	\begin{align}\label{NmB2}
		\cD_{s} \cong  \cM_s^s \left[ \cD_{\ni u} \right]_{u\in \suc(s)}
		\end{align}
	On the other hand, notice that the induced substructure on $\bigcup_{u\in \cM'_s} \cD'_{\ni u}\subseteq \cD_{\ni s}$ is isomorphic to ${\cM'_s}^s\left[ \cD'_{\ni u} \right]_{u\in \cM'_s}$, which in turn, by \Cref{NmB2} is isomorphic to $\cD_s\cong \cD_{\ni s}$. So $\cD_{\ni s}$ contains an isomorphic monochromatic-blue copy of itself, by contradiction to the induction hypothesis.
	\item So $R(s)$ contains an isomorphic copy of $\cM_s$, denoted by $\cM'_s$. Let $\theta_s:\cM'_s\to \cM_s$ be such an isomorphism.
	To conclude we define $S_{n+1}:=\bigcup_{s\in S_{n}} \cM'_s$ and $\theta_{n+1}:=\bigcup_{s\in S_{\alpha}} \theta_{s}$
\end{itemize}

	If $S = \bigcup_{n<\omega} S_n$ and $\theta = \bigcup_{n<\omega} \theta_{\alpha}$, then by its construction $\theta: S\to T$ is an isomorphism of trees such that $\theta \upharpoonright \cM'_s : \cM'_s \to \cM_{\theta(s)}$ is an isomorphism of $\cL$ structures. Since $C(s) =\red$ for all $s\in S$, by definition of $C$, there is a countable dense monochromatic-red $\cD_2\subset \prod_S \cM_s$.
\end{proof}

\begin{theorem}\label{rigidInd}
	There is a countable rigid elementarily indivisible structure, in a finite language.
\end{theorem}
For the proof of \Cref{rigidInd}, we will need the following:
\begin{fact}[\text{\cite[Lemma 3.5]{Me}}]\label{elemEquivResult}
	If $\cM \sim_e \cN$  then
	$\cM$ is elementarily indivisible iff $\cN$ is elementarily indivisible.
\end{fact}
\begin{fact}[\text{\cite[Lemma 3.10]{Me}}]\label{infChain}
	There is a sequence $\{\cA_i\}_{i\in \omega}$  of pairwise-non-isomorphic countable elementarily indivisible structures, in a finite language, such that $\cA_i\prec \cA_j$ for all $i,j\in \omega$. Furthermore, $\cA_0$ can be chosen to be homogeneous.
\end{fact}

\begin{proof}[proof of \Cref{rigidInd}]
	We first give an example in an infinite language and then present a structure in a finite language that is interdefinable with the first, i.e., a structure on the same underlying set with the same $\emptyset$-definable sets.
	
	For the first example, in an infinite language: Let $\{\cA_i\}_{i\in \omega}$ be a set of pairwise-non-isomorphic countable elementarily indivisible structures, in a finite language $\cL$, such that $\cA_i\sim_e \cA_j$ for all $i,j\in \omega$ such that $\cA_0$ is homogeneous, as provided by \Cref{infChain}. 
	Let $ T=\omega^{<\omega}$. Let $\Braket{ \sigma_i | i\in \omega}$ be an enumeration of $T$. Let $\cM_a^0 := \cA_0$ for all $a\in T$ and $\cN_{\sigma_i}^0 := \cA_i$. For all $a\in T$, let $\cM_a$ and $\cN_a$ be expansions of $\cM_a^0$ and $\cN_a^0$, respectively, to a new binary relation $R$ such that $\cM_a$ and $\cN_a$ both interpret $R$ as a full subgraph whenever $\height(a)$ is even and as an empty subgraph whenever $\height(a)$ is odd.  
	Let $\cD'\subset \infProds$ be countable and dense.
	By \Cref{infProdIndOmega}, $\cD'$ is elementarily indivisible. By \Cref{infProdElemEmbed}, $\prod_T^s\cM_t \sim_e \prod_T^s\cN_t$.
	By \Cref{allallElemEmbedBoth} there is a countable dense  elementary substructure $\cD\prec \infProdNs$ such that $\cD'\sim_e\cD$. Since $\cD'$ is elementarily indivisible, so is $\cD$, by \Cref{elemEquivResult}. Now $\cD$ is rigid since if there are distinct $a, b\in \cD$ and $\sigma\in \aut(\cD)$ such that $\sigma(a)=b$, since $a\neq b$, there is some $i<\omega$ such that $\neg s_i(a,b)$ but $\sigma$ sends the $s_i$-equivalence class of $a$ to the $s_i$-equivalence class of $b$, but, by definition of $\cD$, no two $s_i$-equivalence classes are isomorphic.
	
	For an example in a finite language, we notice that $s_i$ is definable from $R$, for all $1\leq i<\omega$:
	\begin{itemize}
		\item $s_1(x,y)\leftrightarrow \bigg(\neg R(x,y) \lor  \exists z \Big(\neg R(x,z)\land \neg R(y,z)\Big)\bigg)$
		\item $s_{2n}(x,y)\leftrightarrow \Bigg( s_{2n-1}(x,y) \land \bigg(R(x,y) \lor \exists z \Big(s_{2n-1}(x,z) \land R(x,z)\land R(y,z)\Big)\bigg)\Bigg)$ for $n\geq 1$.
		\item  $s_{2n+1}(x,y)\leftrightarrow \Bigg( s_{2n}(x,y) \land \bigg(\neg R(x,y) \lor \exists z \Big(s_{2n}(x,z) \land \neg R(x,z)\land \neg R(y,z)\Big)\bigg)\Bigg)$ for $n\geq 1$.
		
	\end{itemize}
	So $\cD$ and $\cD\upharpoonright \cL\cup\{R\}$ are inter-definable, the latter being in a finite language.
\end{proof}

\subsection{$\cL_{\omega_1,\omega}$-elementary indivisibility and transitivity}\label{LWWIndivisibility}
In this \namecref{LWWIndivisibility}, we strengthen the notion of elementary indivisibility to $\cL_{\omega_1,\omega}$ and show that not only does \Cref{rigidInd} fail in this context, but in fact, every $\cL_{\omega_1,\omega}$-elementarily indivisible structure is transitive.

\begin{definition}
A relational structure is \emph{$\cL_{\omega_1,\omega}$-elementarily indivisible} if for every colouring of its universe in two colours, there is a monochromatic $\cL_{\omega_1,\omega}$-elementary substructure isomorphic to it.
\end{definition}
\begin{lemma}\label{LWW one type}
	If $\cM$ is a countable $\cL_{\omega_1,\omega}$-elementarily indivisible structure then $a\equiv_{{\omega_1,\omega}} b$ for any two singletons $a,b\in \cM$.
\end{lemma}
\begin{proof}
		If not, then there is an $\cL_{\omega_1,\omega}$-formula in one free variable $\phi(x)$ such that $\cM \models \exists x\, \phi(x)$ and $\cM\models \exists x\, \neg\phi(x)$. Let $c:M\to\{\text{red},\text{blue}\}$ be defined as
	\[ c(x):=\twopartdef{\text{blue}}{\cM\models \phi(x)}{\text{red}}{\cM\models \neg\phi(x).} \]
	Clearly, no $c$-monochromatic substructure is $\cL_{\omega_1,\omega}$-elementary.
\end{proof}
\begin{theorem}
	Every countable $\cL_{\omega_1,\omega}$-elementarily indivisible structure is transitive.
\end{theorem}
\begin{proof}
	Let $\cM$ be an $\cL_{\omega_1,\omega}$-elementarily indivisible structure and let $a,b\in \cM$ be singletons, then by \Cref{LWW one type}, $a\equiv_{{\omega_1,\omega}} b$. By Scott's Isomorphism Theorem (\cite{Scott}, \cite[Corollary 3.5.4]{Hodges93}), $\Braket{\cM,a}\cong \Braket{\cM,b}$ (where $\Braket{\cM,a}, \Braket{\cM,b}$ are expansions of $\cM$ by a constant symbol for $a,b$ respectively). Finally, any isomorphism between $\Braket{\cM,a}$ and $\Braket{\cM,b}$ is an automorphism of $\cM$ sending $a$ to $b$.
\end{proof}

\bibliographystyle{alpha}
\bibliography{infprodref}

\end{document}